\documentclass[11pt]{article}
\usepackage{fullpage}
\usepackage{amsthm}
\usepackage{comment}
\usepackage{amsmath}
\usepackage{amssymb}
\usepackage{algorithm}
\usepackage{algorithmic}
\usepackage{bm}
\usepackage{comment}
\usepackage{hyperref}
\usepackage{systeme}
\usepackage{graphicx}
\usepackage{float} 
\usepackage{cleveref}

\usepackage{breakcites}
\usepackage{mathtools}

\usepackage{enumitem}
\newtheorem{theorem}{Theorem}[section] 
\newtheorem{lemma}[theorem]{Lemma}
\newtheorem{proposition}[theorem]{Proposition}
\newtheorem{corollary}[theorem]{Corollary}

\newtheorem{fact}[theorem]{Fact}
\theoremstyle{definition}
\newtheorem{definition}[theorem]{Definition}
\newtheorem{remark}[theorem]{Remark}

\usepackage{thmtools}
\usepackage{thm-restate}

\newcommand{\CML}{\mathrm{CML}}

\newcommand{\FF}{\mathbb{F}}
\newcommand{\PP}{\mathbb{P}}
\newcommand{\kfield}{\tilde{k}}
\newcommand{\Kfield}{\tilde{K}}

\newcommand{\Q}{\mathbb{Q}}

\newcommand{\tV}{\tilde{V}}
\newcommand{\tk}{\tilde{k}}

\newcommand{\tX}{\tilde{X}}

\DeclareMathOperator{\Char}{char}

\newcommand{\V}{\mathbf{V}}

\usepackage{comment} 
\newif\ifai

\usepackage{todonotes} 
\aitrue 

\ifai
  \includecomment{ai}
\else
  \excludecomment{ai} 
\fi

\title{$R$-equivalence on Cubic Surfaces I:\\
Existing Cases with Non-Trivial Universal Equivalence}
\begin{ai}
\author{Dimitri Kanevsky, Julian Salazar, and Matt Harvey}
\end{ai}
\date{}

\begin{document}
\newcommand{\Addresses}{{
  \footnotesize
  \textsc{Google DeepMind, 1600 Amphitheatre Parkway, Mountain View, CA 94043}\par\nopagebreak
  \textit{E-mail:} \texttt{\{dkanevsky,julsal,mattharvey\}@google.com}
  \bigskip
}}
\maketitle
\begin{ai}
\Addresses
\end{ai}

\begin{abstract}
\footnotesize
Let $V$ be a smooth cubic surface over a $p$-adic field $k$ with good reduction. Swinnerton-Dyer (1981) proved that $R$-equivalence is trivial on $V(k)$ except perhaps if $V$ is one of three special types---those whose $R$-equivalence he could not bound by proving the universal (admissible) equivalence is trivial. We consider all surfaces $V$ currently known to have non-trivial universal equivalence.  Beyond being intractable to Swinnerton-Dyer's approach, we observe that if these surfaces also had non-trivial $R$-equivalence, they would contradict Colliot-Thélène and Sansuc's conjecture regarding the $k$-rationality of universal torsors for geometrically rational surfaces.\\
\indent By devising new methods to study $R$-equivalence, we prove that for 2-adic surfaces with all-Eckardt reductions (the third special type, which contains every existing case of non-trivial universal equivalence), $R$-equivalence is trivial or of exponent 2. For the explicit cases, we confirm triviality: the diagonal cubic $X^3+Y^3+Z^3+\zeta_3 T^3=0$ over $\mathbb{Q}_2(\zeta_3)$---answering
a long-standing question of Manin's (\textit{Cubic Forms}, 1972)---and the cubic with universal equivalence of exponent 2 (Kanevsky, 1982).\\
\begin{ai}
\indent This is the first in a series of works derived from a \textit{year} of interactions with generative AI models such as AlphaEvolve and Gemini 3 Deep Think, with the latter proving many of our lemmas. We disclose the timeline and nature of their use towards this paper, and describe our broader AI-assisted research program in a companion report (in preparation).
\end{ai}
\end{abstract}


\section{Introduction}

Let $V(k)$ be the set of geometric $k$-points of a geometrically irreducible and reduced projective variety $V$. \textit{$R$-equivalence} is the relation on $V(k)$ where $P \sim Q$ if they are connected by some  chain of rational curves in $V$, i.e., there exists a finite sequence of $k$-points $P, \dotsc, Q$ where each adjacent pair lies in the image of a $k$-morphism from $\PP^1_k$ to $V$. It is pertinent to Diophantine problems---in other words, to describing the set of points $V(K)$ where $K$ is a number field---as it is the finest equivalence relation such that two points given by the same parametric solution to a system of equations defining $V$ are equivalent \cite{diophantine}. For example, because the Brauer-Manin pairing is constant under $R$-equivalence, this relation has been used to prove the sufficiency of the Brauer-Manin obstruction to the Hasse principle \cite[\S VI]{manin3} and to weak approximation \cite{weak-approx} for certain varieties. However, due to the challenge of determining the existence of chains of rational curves, $R$-equivalence has historically resisted computation.

Despite this difficulty, cubic hypersurfaces $V$ offer a tractable setting for arithmetic geometry due to the additional algebraic structure of their $k$-points. When $\dim V = 1$, one retrieves the rich study of points on elliptic curves and their abelian group structure, which expresses geometrically as the chord-and-tangent process. However, applying this geometric process to $\dim V > 1$ does not give a well-defined binary operation due to e.g., the presence of lines on $V$. Manin \cite{manin1968cubic} proposed to modulo out this difficulty, deriving the study of \textit{admissible point classes} on cubic surfaces and their commutative Moufang loop (CML) structure (\Cref{sec:background}), which can be non-associative \cite{dknonassoc}. Thus for cubic hypersurfaces, proving $R$-equivalence is admissible guarantees it has a CML structure; then, computing the universal admissible equivalence provides an upper bound. This motivation, combined with Manin's geometric formulation of admissibility, led Swinnerton-Dyer \cite{swinnertondyer} to investigate this finer universal equivalence for smooth cubic surfaces $V$ over local fields $k$ with good reduction—i.e., where the surface is defined by a cubic equation with integer coefficients whose reduction modulo $p$ defines a smooth cubic surface—by lifting where possible from their reductions over finite fields. In this way he showed all points in $V(k)$ are universally equivalent and thus $R$-equivalent, except possibly when the characteristic of the residue field $\tilde{k}$ is 2 or 3 and conditions on the special fiber $\tilde{V}$ (i.e., $V$'s reduction mod $\mathfrak{p}$) are exceptional; we restate his Theorems 1 and 2, for finite and local fields respectively, in \Cref{sec:swd-background}.

This first work considers the possibility of non-trivial $R$-equivalence arising from the cases currently known to have non-trivial universal equivalence \cite{dimitrikanevsky}. 
The interest in computing $R$-equivalence for these surfaces stems from their unique position in the arithmetic landscape: because these surfaces possess smooth reduction, one expects their Brauer equivalence---where two points are equivalent if they give the same output under the Brauer-Manin pairing for all elements of the Brauer group $\text{Br}(X)$---to be trivial. More precisely, let $S$ be the Néron-Severi torus of $X$. For any smooth, projective, geometrically rational surface $X$ over a characteristic 0 field $k$ with a $k$-point, the map $A_0(X) \to H^1(k, S)$ is injective \cite[Prop.4]{colliot1983}; furthermore, in the case of smooth reduction over a $p$-adic field, this map is zero \cite[Thm.0.2]{bloch1981chow} \cite[Thm.4]{dalawat2005groupe} and therefore $A_0(X)$ is trivial. The Brauer-Manin pairing respects rational equivalence, so this implies trivial Brauer equivalence.
This creates a compelling opening: if $R$-equivalence were non-trivial in these cases, it would be strictly finer than rational and thus Brauer equivalence. This would give the first example of a smooth, projective, geometrically rational surface $X$ over a field of characteristic zero where the map $X(k)/R \to CH_0(X)$ is not injective. Such a result would be particularly noteworthy as it relates to a long-standing question posed by Colliot-Thélène and Sansuc: are universal torsors over such surfaces $k$-rational as soon as they possess a $k$-point \cite[\S 2.8]{colliot1987descente}? If so, the descent map $X(k)/R \to H^1(k, S)$, which factors through $CH_0(X)$, must be injective---a contradiction.

While we ultimately do not find non-trivial $R$-equivalence in this work—consistent with the conjecture—testing these boundaries remains a primary motivation for the new $R$-equivalence methods developed here. Furthermore, resolving $R$-equivalence for these cases is compelling in its own right, as they represent a gap in our methods to compute $R$-equivalence and exceptions to what would otherwise be a clean statement on cubic surfaces over local fields (\Cref{thm:swd-2}). In particular, Manin \cite[end of \S II.6]{manin3firstedition} asked about the universal and $R$-equivalence of the projective cubic
\begin{equation}
\label{eq:manin}
    V : X^3 + Y^3 + Z^3 + \theta T^3 = 0 \textrm{\quad over } \mathbb{Q}_2(\theta), \textrm{ where } \theta^2 + \theta + 1 = 0,
\end{equation}
i.e., $\theta$ is a primitive third root of unity ($\zeta_3$). While its universal equivalence was established as $\mathbb{Z}/3 \times \mathbb{Z}/3$ early on \cite[1.6]{dimitrikanevsky}, its $R$-equivalence  has remained open \cite[Thm.2.iii]{swinnertondyer} \cite[\S16]{manin3} \cite[\S6]{weak-approx} despite advances in non-constructive determination of $R$-equivalence (\Cref{sec:timeline}).

\subsection{Main Results}

First, we outline currently known situations of smooth cubic surfaces $V$ over $p$-adic field $k$ with good reduction that guarantee non-trivial universal equivalence. Note that these fall under case (iii) of \cite[Thm.~2]{swinnertondyer} (\Cref{thm:swd-2}):
\begin{enumerate}
    \item $\tilde{V}(\tilde{k})$ is line-free and all-Eckardt, with more than one point \cite[Cor.~2.5]{dimitrikanevsky}. \Cref{eq:manin} is the explicit case of this \cite[Ex.~II.6.5]{manin3firstedition} \cite[Ex.~1.6.i]{dimitrikanevsky}.
    \item $V : F_{1,1,1}(T) = 0$ from \Cref{eq:main}. This is a case of $\tilde{V}(\tilde{k})$ line-free and a single (Eckardt) point \cite[\S 3.4]{dimitrikanevsky}, which allows but does not guarantee non-trivial equivalence.
    \begin{align}
\label{eq:main}
\begin{split}
F_{b_0, b_1, b_2}(T) &= T_0^2T_1 + T_0T_1^2 + T_2^3 + T_2^2T_3 + T_3^3 + T_1(T_1^2+T_2T_3+T_2^2+T_3^2)\\
&\quad\quad+ 2T_0(b_0T_2^2+b_1T_2T_3+ b_2T_3^2).
\end{split}
\end{align}
\end{enumerate}
In \Cref{sec:3-torsion-free}, we prove the following general result that covers both cases:
\begin{restatable}{maintheorem}{MainResultOne}
\label{thm:main}
\label{thm:eckardt_conditions}
Let $V$ be a smooth cubic surface over a $2$-adic local field $k$
with good reduction such that points $\tilde{V}(\tilde{k}) \ne \emptyset$ are all Eckardt.
Then $R$-equivalence on $V(k)$ is trivial or of exponent 2.
\end{restatable}
Its proof primarily relies on the application of Manin's norm map for quadratic extensions. This map preserves the period-three components within the loop of universal equivalence classes on the surface. From it, we can easily specialize to the broad first case:
\begin{corollary}
\label{thm:manin}
In addition, if $\tilde{V}(\tilde{k})$ is line-free, is comprised of more than one Eckardt point, and $V(k)$ contains an Eckardt point, then $R$-equivalence is trivial. In particular, $V : X^3 + Y^3 + Z^3 + \zeta_3 T^3 = 0$ over $\mathbb{Q}_2(\zeta_3)$ has trivial $R$-equivalence.
\end{corollary}
\begin{proof}
We meet the conditions of \cite[Cor.~1.5]{dimitrikanevsky}, giving universal equivalence that is of exponent 3. $R$-equivalence is coarser, so it cannot be exponent 2. Therefore it must be trivial. One can verify that Manin's diagonal cubic meets all conditions; see \cite[\S 8]{swinnertondyer} and \cite[Ex.~1.6.i]{dimitrikanevsky}.
\end{proof}

However, Manin's original norm construction \cite[\S 15]{manin3} is not applicable to the period-two components of this loop. In \Cref{sec:1pt-r-equiv}, we extend the method to encompass these period-two components. Specifically, we consider a specialization
of the other published 2-adic case, from \cite[3.4]{dimitrikanevsky}, where universal equivalence is known to be exponent two. Let $V$ be the cubic surface over $\mathbb{Q}_2$ defined by $F_{b_0, b_1, b_2} = 0$ from \Cref{eq:main} where $b_i = 1 \mod 2$ for all $i$ (slightly generalizing the second definitive case). Modulo 2 it gives the same surface with one point over $\mathbb{F}_2$ that was considered in \cite{swinnertondyer}). The work demonstrated an admissible equivalence on $V(\mathbb{Q}_2)$ with a non-trivial period-two component, giving a lower bound on the structure of universal equivalence. Using our extended method, we explicitly compute the universal and $R$-equivalence of this example:

\begin{restatable}{maintheorem}{MainResultTwo}
\label{single}
Let $V$ defined by \Cref{eq:main} be smooth such that $\tilde{V}$ has exactly one point. Then $V(\mathbb{Q}_2)$ constitutes a single class of $R$-equivalence and two classes of universal equivalence.
\end{restatable}
\noindent In all, we use new and expanded methods to resolve two long-standing cases of exception (iii) to \cite{swinnertondyer}'s Theorem 2 (reproduced as \Cref{thm:swd-2}).

\begin{ai}
Finally, in \Cref{sec:ai-assist} we disclose the non-trivial role played by generative artificial intelligence (AI) in completing this work. Some distinctive aspects as of March 2026 include the relative notability of the problem (\Cref{sec:timeline}), reliance on literature and methods largely unavailable online, and the unexpected role of AI in \textit{improving} rigor in a domain that tends towards high-level geometric reasoning. We also include a timeline of AI use, occurring over a year across multiple models---longer than other AI-assisted results---glimpsing how AI plays into the messiness of long-term research. Finally, our background is atypical: instead of being full-time mathematicians coordinating with AI researchers, we are all career AI researchers who were full-time mathematicians 9+ years ago (though the first author, a student of Manin's, has continued to publish in the field \cite{dknonassoc, dknonassoc2}).

\begin{remark}
This is only the first publication atop a growing body of theory and results over a year (thus far). We propose and discuss this paradigm of AI-assisted mathematical \textit{theory building} in a forthcoming companion report \cite{aireport}.
\end{remark}

\noindent \textbf{Acknowledgements:} We thank Jean-Louis Colliot-Thélène for his correspondence on this subject. We extend our gratitude to our colleagues at Google DeepMind, particularly Lucas Dixon for his encouragement and assistance in utilizing these systems, Daniel Zheng for his helpful comments that improved the paper's content, Adam Zsolt Wagner for onboarding us to AlphaEvolve, and Martin Wattenberg, Eric Wieser, and many others who provided useful discussions and shared promising mathematical tools and agents we hope to leverage in future work.
\end{ai}

\subsection{Timeline of Prior Work}
\label{sec:timeline}
We give a brief history of $R$-equivalence on $p$-adic cubic hypersurfaces, and the question of exceptional cases, for generalist readers:

\begin{itemize}
  \item \textbf{1968:} The theory of CMLs of admissible equivalence classes on cubic hypersurfaces, generalizing the chord-tangent construction of elliptic curves, is introduced by Manin \cite{manin1968cubic}.
  \item \textbf{1972:} Manin uses this theory to prove $R$-equivalence is finite for smooth cubic hypersurfaces over local fields. The cubic surface of \Cref{thm:manin} is shown to have a non-trivial admissible equivalence over the 2-adics by Manin, who then asks about its $R$-equivalence \cite[\S II.6]{manin3firstedition}.
  \item \textbf{1981:} Swinnerton-Dyer \cite{swinnertondyer} shows $R$-equivalence is trivial for many $p$-adic cubic surfaces with good reduction by bounding them via their trivial universal admissible equivalence. However, as Manin's exceptional surface has a non-trivial admissible equivalence, it is not covered by \cite{swinnertondyer}'s methods, falling into exception (iii) of Theorem 2 (in general, the exceptions occur in $p=2, 3$). \cite[\S 8]{swinnertondyer} references Manin and his surface and resolves $R$-equivalence for the finite-field reduction of the surface, but not the surface itself.
  \item \textbf{1982:} Kanevsky \cite{dimitrikanevsky} proves that Manin's admissible equivalence is universal and shows another exception (iii) to Theorem 2---the one-point reduction case---has a non-trivial admissible equivalence.
  \item \textbf{1984:} Kanevsky's preprint \cite{dkbrauer} claims to show Manin's surface has trivial $R$-equivalence.
  \item \textbf{1986:} \cite{manin3} acknowledges some results from \cite{dkbrauer} but not its proof of $R$-equivalence for his surface, restating the question as open.
  \item \textbf{1987:} In their treatise \cite{colliot1987descente} on the theory of descent on rational varieties, Colliot-Th\'{e}l\`{e}ne and Sansuc note their theory's reliance on ``la premi\`{e}re hypoth\`{e}se sur les torseurs universels,'' for which no smooth, projective, $k$-rational counterexamples exist. Weaker versions of the hypothesis are known to fail, but not for surfaces.
  \item \textbf{1999:} Using deformation theory, Kollar \cite{kollar1999rationally} proved that $R$-equivalence is finite for rationally connected varieties over local fields, generalizing Manin's finiteness result via different means. This result was published in \textit{Annals} (the top journal in mathematics, which per \cite{feng2026aletheia}'s human-AI taxonomy is an indicator of a Level 3 ``Major Advance''), to help readers place the historical significance of this line of work.
  \item \textbf{2001:} Swinnerton-Dyer \cite[\S 6]{weak-approx} relates $R$-equivalence to the long-standing question of weak approximation (i.e., when a cubic surface over the $p$-adics approximates its points over a number field). However, he notes his methods still cannot lift his $R$-equivalence proof for the finite-field reduction to compute $R$-equivalence for Manin's surface.
  \item \textbf{2003, 2008:} New results prove trivial $R$-equivalence on more $p$-adic cubic hypersurfaces. However, these results are for hypersurfaces with large residue fields \cite{kollar2003rational} or in higher dimensions \cite{madore2003equivalence, madore2008equivalence}, due to the limitations of e.g., deformation-theoretic approaches.
\end{itemize}

\section{Background}
\label{sec:background}

In general $V$ will denote a geometrically irreducible and reduced projective cubic hypersurface defined over a field $k$. As many of the statements below are in publications not easily accessible, we summarize the known results needed here.

\subsection{Algebraic Structures on Cubic Hypersurfaces}

Here we follow the approach first introduced in \cite{manin1968cubic} and refined in \cite{manin3}:

\begin{definition}
Three points $P_1, P_2, P_3 \in V(k)$ are \textit{collinear} if they lie on a straight line $L$ (as understood in a projective space containing $V$) that is defined over $k$, and either $L \subset V$ or $P_1+P_2+P_3$ is the intersection cycle $V \cdot L$.
\end{definition}

Hence, for any $P_1,P_2 \in V(k)$ there exists $P_3 \in V(k)$ such that $P_1,P_2,P_3$ are collinear, and if $P_1 \neq P_2$ and there is no straight line on $V$ defined over $k$ through $P_1, P_2$, then $P_3$ is uniquely defined. For such pairs of points we can define $P_1 \circ P_2 := P_3$.

\begin{definition}
An equivalence relation $A$ on $V_r(k)$ (which denotes non-singular points of $V(k)$) is said to be \textit{admissible} if when $P_1, P_2, P_3$ and $P_1', P_2', P_3'$ are collinear triples such that $P_1 \sim_A P_1'$ and $P_2 \sim_A P_2'$, then $P_3 \sim_A P_3'$.
\end{definition}
\noindent
Admissible relations are those which turn collinearity into a binary operation on cubic hypersurfaces. 
If $[P]$ is the equivalence class of $P$ in $V_r(k)/A$, where $A$ is an admissible equivalence relation, then the operation $[P_1] \circ [P_2]$ is well-defined and can be characterized as follows:
\begin{fact}[Collinearity as binary operation]
\label{item}
\mbox{}
\begin{itemize}
    \item If $P_1 \neq P_2$ and there is no line through these points lying on $V$, then $P_1 \circ P_2$ is uniquely determined as the third point of intersection of the line $\langle P_1, P_2 \rangle$ with $V$. Its class is denoted $[P_1 \circ P_2]$.
    \item If $P_1$ and $P_2$ lie on a line $L$ belonging to $V$, then admissibility necessitates that $L \cap V_r(k)$ collapses into a single class; thus, we define $[P_1] \circ [P_2] := [P_1] = [P_2]$.
    \item If $P_1 = P_2$ and there is no line through $P_1$ lying on $V$ over $k$, let $C$ be the tangent section (the intersection of $V$ with the tangent plane at $P_1$). Then all points on $C$ distinct from $P_1$ belong to the same admissible equivalence class. Furthermore, if $C$ possesses a rational tangent line at $P_1$, then all points on $C$ (including $P_1$ itself) belong to the same admissible equivalence class.
\end{itemize}
\end{fact}

\begin{proposition}
For admissible $A$ and any $S_1, S_2 \in V_r(k)/A$, there is a uniquely determined class $S_3 \in V_r(k)/A$ such that there exist collinear points $P_1, P_2, P_3$ with $P_i \in S_i$. By defining a new binary operation $S_1 S_2 := [O] \circ (S_1 \circ S_2)$ where $O$ is a fixed point, one converts $V_r(k)/A$ into a \textit{commutative Moufang loop (CML)} with $[O]$ as the identity.
\end{proposition}
\noindent CMLs are a non-associative generalization of abelian groups. For brevity, we write
\[
\mathcal{M}_A(k) := \CML(V_r(k)/A, [O])
\]
for admissible $A$, as $V$ is implicit and the CML structure is independent of the base point $[O]$ \cite[5.1]{manin3}, generalizing the result for the geometric group law on elliptic curves.

\begin{remark}
We write ``$A$-equivalence'' to refer to either the equivalence relation $A$, or the (isomorphism class of) CMLs derived from admissible $A$, based on context.
\end{remark}

Finally, we can take the initial object of all such admissible quotients:

\begin{definition}
\textit{Universal equivalence} is the finest admissible equivalence relation on $V_r(k)$.
\end{definition}
\begin{proposition}
The universal equivalence exists and is unique, with all admissible relations being quotients of it. When $[O]$ is fixed, there is a natural surjection of CMLs ($\mathcal{M}_U(k) \to \mathcal{M}_A(k)$).
\end{proposition}

\subsection{Properties of $R$-equivalence for $\dim V \ge 2$}

We use the definition of $R$-equivalence that we gave at the beginning of the introduction section.
One can find some different equivalent definitions of $R$-equivalence in \cite[14.1-2]{manin3}.
$R$-equivalence is defined for any algebraic variety $V$ over a field $k$. However, additional algebraic structure comes under very mild cubic hypersurface conditions on $V$.

\begin{theorem}[{\cite[Thm.~1]{kollar2002unirationality}}]
\label{thm:k-point}
Let $V$ be smooth with $\dim V \ge 2$. Then $V(k) \neq \emptyset$ (i.e., $V$ contains a $k$-point) if and only if $V$ is unirational over $k$, i.e., there exists a dominant rational map $f : \mathbb{P}^n_k \dashrightarrow V$ defined over $k$.
\end{theorem}


\begin{remark}
\label{rem:no-general-type}
The condition ``$k$-point of general type'' was almost exclusively used by Manin to establish $k$-unirationality via \cite[Thm.12.11]{manin3}. Likewise, Kanevsky only uses the existence of a $k$-point of general type to use one of Manin's unirationality results, namely that admissible equivalence classes are open in the $k$-topology for local fields $k$ \cite[2.7]{dimitrikanevsky}. Hence, via \Cref{thm:k-point} we can drop the ``general type'' part moving forward, citing this remark when we have done so.
\end{remark}

\begin{proposition}[{\cite[Thm.~14.3]{manin3}}]
\label{thm:r-equiv-admissible}
If $V$ is unirational and $k$ is an infinite field, then $R$-equivalence on $V_r(k)$ is admissible.
\end{proposition}

Recall that admissible equivalences collapse lines in $V$ into classes; hence, so does the universal admissible equivalence. Intuitively, $R$-equivalence collapses rational curves in $V$ into classes, resulting in a coarser admissible relation.


In the case of cubic surfaces and higher-dimensional cubic hypersurfaces, the algebraic structure induced via collinearity as a binary operation is more restricted than in elliptic curves.

\begin{proposition}[{\cite[Cor.~13.3]{manin3}} and \Cref{rem:no-general-type}]
\label{thm:structure}
Let $A$ be admissible on $V_r(k) \ne \emptyset$. If the field $k$ is infinite and $\dim V \ge 2$, then $\mathcal{M}_A(k)$ is the direct product of an abelian group of exponent 2 and a CML of exponent 3.
\end{proposition}

\subsection{Results from \cite{swinnertondyer}}
\label{sec:swd-background}

Let $\tilde k = \mathcal{O}_k / \mathfrak{p}_k$ denote the residue field of $k$ (in general, a tilde denotes reduction mod $\mathfrak{p}_k$). 

\begin{definition}
A (geometric) \textit{Eckardt point defined over $k$} on a cubic surface $V$ is a point $P \in V(k)$ where $V \cdot T_P V$ decomposes into three straight lines through $P$ in the algebraic closure $\bar{k}$. 
\end{definition}

Here, we reproduce the key results and intermediate findings of \cite{swinnertondyer} in our notation. 

\begin{definition}
The \textit{normalized Hessian} $H^*$ of a cubic form $F(X, Y, Z, T) = 0$ is given by
\[
(1.1) \quad H^*(X, Y, Z, T) = \frac{1}{4}
\begin{vmatrix}
F_{XX} & F_{XY} & F_{XZ} & F_{XT} \\
F_{YX} & F_{YY} & F_{YZ} & F_{YT} \\
F_{ZX} & F_{ZY} & F_{ZZ} & F_{ZT} \\
F_{TX} & F_{TY} & F_{TZ} & F_{TT}
\end{vmatrix}
\]
where the subscripts denote derivatives.
\end{definition}

\begin{theorem}[{\cite[Thm.~1]{swinnertondyer}}]
\label{thm:swd-1}
Let $V$ be a nonsingular cubic surface defined over the
finite field of $q$ elements, and suppose that $V$ contains $n$ rational points;
then all these points belong to the same class for universal equivalence
except when $V$ contains no rational line and all its rational points are
Eckardt. In the exceptional case there are n classes each consisting of
one point; and either $q = 2$, $n = 3$ or $q = 4$, $n = 9$.
\end{theorem}

\begin{theorem}[{\cite[Thm.~2]{swinnertondyer}}]
\label{thm:swd-2}
Let $V$ be a nonsingular cubic surface with equation $$F(X, Y, Z, T) = 0,$$
where the coefficients of $F$ are integers in a $p$-adic field $k$; then $\tilde{V}$ is given by $\tilde{F}=0$
and we assume that this equation defines a nonsingular cubic surface (good reduction). All the rational
points of $V$ belong to the same class for universal equivalence except perhaps in the following three cases:
\begin{enumerate}[label=(\roman*)]
    \item $\Char{\tilde{k}} = 3$ and the surface $\tilde{H}^* = 0$ touches $\tilde{V}$ at every rational point of $\tilde{V}$.
    \item $\Char{\tilde{k}} = 2$ and $\tilde{H}^*$ vanishes.
    \item $\Char{\tilde{k}} = 2$ and every rational point of $\tilde{V}$ is an Eckardt point.
\end{enumerate}
In particular \cite[\S 5]{swinnertondyer} gives an explicit list of representatives for reductions satisfying (iii).
\end{theorem}

\begin{lemma}[{\cite[Lem.~14]{swinnertondyer}}]
\label{thm:swd-14}
Let $\tilde{P}_1, \tilde{P}_2, \tilde{P}_3 \in \tV$ defined over $\tk$ be collinear; they can be lifted to points $P_1, P_2, P_3 \in V$ defined over $k$ which are collinear in the same sense. 
\end{lemma}

\subsection{Results under field extensions}

We now describe the behaviour of $R$-equivalence classes $V_r(k)/R$ under a quadratic extension $K$ of the base field, where there exists a \textit{norm map} that passes from $K$ to $k$:

\begin{proposition}[{\cite[Prop.~15.1-15.1.1]{manin3}}]
\label{thm:inverse-map}
Let $V$ be unirational and $[K:k]=2$ where $K$ is separable over an infinite field $k$. Let $i: V_r(k)/R \to V_r(K)/R$ be the morphism mapping classes to classes. Then there exists a map of sets $N : V_r(K)/R  \to V_r(k)/R$ such that $N(i(S)) = S \circ S$. Furthermore, viewing $i$ as a CML morphism:
\[
S \in \text{ker}(i) \implies S^2 = [O] \circ (S \circ S) = [O] \circ N(i(S)) = [O] \circ N(i([O])) = [O] \circ ([O] \circ [O]) = [O],
\]
i.e., the kernel of $i$ only contains elements of order 2.
\end{proposition}

\cite[Prop.~5.1]{dkbrauer} implicitly turns \Cref{thm:structure} into a specialized version of the following structure result:

\begin{proposition}
\label{thm:RU}
Let $V$ be unirational with $\dim V \ge 2$ over an infinite field $k$. Let $K$ be a tower of separable quadratic extensions over $k$. Then:
\begin{equation}
\label{eq_RU}
    \#\mathcal{M}_R(k)[3] \leq \#\mathcal{M}_R(K)[3] \leq \#\mathcal{M}_U(K)[3]
\end{equation}
where $U$ is the universal equivalence relation and $[3]$ denotes the 3-torsion component of the CML.
\end{proposition}

\begin{proof}
Let $S$ be a class of points on a cubic surface $V(k)$ belonging to the 3-torsion component of the associated CML, i.e., $S \in \mathcal{M}_R(k)[3]$. This is a CML of exponent 3. Let $\ell/k$ be the first separable quadratic extension; by \Cref{thm:inverse-map}, the CML morphism $i: V_r(k)/R \rightarrow V_r(\ell)/R$ is injective. Hence $\#\mathcal{M}_R(k)[3] \le \#\mathcal{M}_R(\ell)[3]$; repeating this reasoning gives $\le \#\mathcal{M}_R(K)[3]$. Then, by \Cref{thm:r-equiv-admissible}, $R$-equivalence is admissible on $V_r(K)$, making it compatible but coarser than universal (admissible) equivalence, and so we have a surjection $V_r(K)/U \to V_r(K)/R$ of CMLs. It follows that $\#\mathcal{M}_U(K)[3] \ge \#\mathcal{M}_R(K)[3]$. \qedsymbol
\end{proof}

\subsection{Lifting Results on $\dim V = 2$}

The primary strategy for equivalence results on local fields has involved describing under which conditions equivalence on finite fields tells us about their local fields.

\begin{definition}
\label{def:general-position}
 We say that $P_1$, $P_2$ are \textit{in general position} when $P_1 \ne P_2$ and the line $P_1 P_2$ is neither tangent to $V$ nor contained in $V$.
\end{definition}

The following facts follow from theorems 11.7, 13.2 and corollaries 6.1.3, 13.3 in \cite{manin3}.\\
The universal equivalence $U$ on $V(k)$ can be split into two admissible equivalences $U_2$ and $U_3$ such that $U = U_2 \cap U_3$, where $|V(k)/U_3| = 3^n$ and $|V(k)/U_2| = 2^m$ for non-negative integers $n, m$. These split equivalences satisfy the following properties:
\begin{itemize}
    \item \textbf{Property 3 ($U_3$):} For any class $X \in V(k)/U_3$, we have $X \circ X = X$.
    \item \textbf{Property 2 ($U_2$):} There is a class $X_0 \in V(k)/U_2$ such that for any class $X \in V(k)/U_2$, we have $X \circ X = X_0$.
\end{itemize}
Because $U$ is the overall finest admissible equivalence, $U_i$ (for $i \in \{2, 3\}$) is strictly the \emph{finest} admissible equivalence on $V(k)$ satisfying the respective Property $i$. We denote the corresponding finest admissible equivalence satisfying Property $i$ on $\tilde{V}(\tilde{k})$ by $\tilde{U}_i$.

\begin{definition}[$i$-Class Free Point]
Let $i \in \{2, 3\}$. A point $\tilde{P} \in \tilde{V}(\tilde{k})$ is called \emph{$i$-class free} if all points on $V(k)$ lifted from $\tilde{P}$ belong to the exact same class $U_i$.
\end{definition}

Then:
\begin{proposition}[{\cite[Prop.~2.4]{dimitrikanevsky}} and \Cref{rem:no-general-type}]
\label{thm:class-free}
Let $V(k) \ne \emptyset$ be a nonsingular cubic surface defined over a $p$-adic local field $k$. Suppose upon reduction to the residue field $\tilde{k}$ that no straight lines defined over $\tilde{k}$ and lying on $\tilde{V}$ pass through its Eckardt points. For an Eckardt point $\tilde{P} \in \tilde V(\tilde k)$ let $E_{\tilde{P}}$ denote equivalence classes of points $P$ reducing to $\tilde{P}$. If an Eckardt point $\tilde{P}$ is in general position together with some point of $\tilde{V}(\tilde{k})$, then $\tilde{P}$ is class-free for $\Char{\tilde{k}} \ne 2,3$,  $\CML(E_{\tilde{P}})$ has period 2 for $\Char{\tilde{k}} = 2$, and $\CML(E_{\tilde{P}})$ has period 3 for $\Char{\tilde{k}} = 3$.
\end{proposition}

\section{Proof of \Cref{thm:main}}
\label{sec:3-torsion-free}
First we prove a few lemmas.
\begin{lemma}
\label{thm:liftlines}
Let $K$ be a local field with ring of integers $\mathcal{O}_K$ and residue field $k$. Let $V \subset \mathbb{P}^3_K$ be a cubic surface defined by a homogeneous polynomial $F$. Assume that $V$ has good reduction, meaning the reduction $\tilde{V}$ is smooth over $k$.

Then, every line $\tilde{L}$ on $\tilde{V}$ lifts uniquely to a line $L$ on $V$.
\end{lemma}

\begin{proof}
It is a standard result that the Fano scheme of lines on a smooth cubic surface is smooth of dimension 0. As noted by Swinnerton-Dyer in the proof of Lemma 14,   \cite{swinnertondyer}, this implies that the system of equations defining the lines has a non-vanishing Jacobian determinant at any solution $\tilde{L}$ on the smooth reduction $\tilde{V}$. Consequently, by the multivariate Hensel’s Lemma, the solution $\tilde{L}$ lifts uniquely to a line $L$ on $V$ defined over $\mathcal{O}_K$.
\end{proof}

\begin{lemma}[{\cite[Prop.~13.7]{manin3}}]
\label{thm:line-implies-3-torsion-free}
Let $V$ be a smooth cubic surface over $k$. If $V$ contains a $k$-rational line, then its universal equivalence group $\mathcal{M}_U(k)$ is $3$-torsion-free (i.e., $\mathcal{M}_U(k)[3] \cong \{1\}$).
\end{lemma}
\begin{proof}
Let $L$ be a $k$-rational line on $V$. Since universal equivalence is admissible, all points in $L(k)$ belong to the same universal equivalence class, which we may denote as the identity class $[e]$. Now, let $P \in V(k)$ be a point not on $L$ and let $[P]$ denote its $U_3$ equivalence class. There exists a $k$-rational plane $\Pi$ containing $L$ and $P$. The intersection of this plane with $V$ consists of the line $L$ and a residual conic $C$ passing through $P$. 

Let $P_0$ be a $k$-rational intersection point of $L$ and the tangent line to $C$ at $P$ (which belongs to $\Pi$).
Since $(P, P, P_0)$ are collinear, $[P] \circ [P] = [P_0] = [e]$.  In the context of the $U_3$ equivalence, this implies $[P] \circ [P] = [P] = [e]$. Since this holds for an arbitrary point $P$, it follows that the 3-torsion is trivial.
\end{proof}

We will also use the following generalized Hasse-Weil bound, due to Aubry and Perret \cite{aubry1996weil} and concurrently Leep and Yeomans \cite{leep1994number}:

\begin{proposition}[Hasse-Weil for singular curves]
\label{thm:aubry-perret}
Let $X$ be an absolutely irreducible projective curve over $\FF_Q$ with arithmetic genus $\pi_X$. Then:
\[ | \#X(\FF_Q) - (Q+1) | \le 2\pi_X \sqrt{Q} \]
\end{proposition}

Since $\pi_X$ accounts for the singularities (specifically $\pi_X = g_{geom} + \delta$, where $\delta$ measures singularity complexity), this bound holds regardless of smoothness.

\begin{lemma}
\label{thm:genposall}
Let $\V \subset \PP^3$ be a smooth cubic surface defined over $\kfield$. For any integer $n \ge 1$, there exists a finite tower of quadratic extensions $\Kfield/\kfield$ such that for any set of $n$ points $P_1, \dotsc, P_n \in \V(\Kfield)$, one can find a point $P' \in \V(\Kfield)$ in general position with respect to $P$.
\end{lemma}

\begin{proof}
We split the proof into a few steps, first establishing the case of $n = 1$, writing $P$ for $P_1$, then unifying over $n$ points.

\vspace{0.5em}
\noindent\textbf{Step 1: Geometric Setup}

Fix a point $P \in \V(\Kfield)$. We identify the set of points $P' \in \V$ that are \textbf{not} in general position with $P$. Let $F(X)=0$ be the equation of $\V$. The intersection of the line $L_{P,P'}$ with $\V$ leads to the condition for the ``bad locus'' $B_P$:
\[ B_P = C_P \cup D_P \]
where:
\begin{enumerate}
    \item $C_P = \V \cap T_P\V$: The intersection of the surface with the tangent plane at $P$. Since $\V$ is a smooth surface, the Zariski tangent space $T_P\V$ is strictly a 2-dimensional plane, making $C_P$ a plane cubic curve (degree 3).
    \item $D_P = \V \cap \{Q \mid \Delta_Q(P)=0\}$: The intersection of the surface with the first polar quadric of $P$. Because $\V$ is absolutely irreducible, this forms a proper space curve of degree $2 \times 3 = 6$.
\end{enumerate}
The total degree of the bad locus is $\deg(B_P) = 9$.

\vspace{0.5em}
\noindent\textbf{Step 2: Point Counting Bounds}

Since $\Kfield$ is finite, we may equally well denote it $\FF_Q$. We must bound $|B_P(\FF_Q)|$.
We decompose $B_P$ into its absolutely irreducible components over $\bar{\FF}_Q$:
\[ B_P = \bigcup_{i=1}^m \Gamma_i \]
where each $\Gamma_i$ is an absolutely irreducible projective curve of degree $d_i$.
Note that $\sum_{i=1}^m d_i = 9$.

Before applying point-counting bounds, we note a technical detail regarding the field of definition. If an absolutely irreducible component $\Gamma_i$ is not defined over $\FF_Q$, its $\FF_Q$-rational points must lie in the intersection of $\Gamma_i$ with its Frobenius conjugates. By B\'ezout's theorem, this intersection is finite and yields at most $d_i^2 \le 81$ points. Since $81 = O(1)$, this is safely and strictly bounded above by the Hasse-Weil bound for singular curves (\Cref{thm:aubry-perret}) for sufficiently large $Q$. 

For the components defined over $\FF_Q$, we bound the number of rational points $|\Gamma_i(\FF_Q)|$ by considering two cases:

\begin{itemize}
    \item \textbf{Case A: $\Gamma_i$ is Smooth.}
If a component $\Gamma_i$ is non-singular, we apply the standard \textbf{Hasse-Weil Bound}.
Let $g_i$ be the geometric genus of $\Gamma_i$.
\[ \left| |\Gamma_i(\FF_Q)| - (Q+1) \right| \le 2g_i \sqrt{Q} \]
Thus,
\[ |\Gamma_i(\FF_Q)| \le Q + 1 + 2g_i \sqrt{Q} \]
    \item \textbf{Case B: $\Gamma_i$ is Singular.}
If a component $\Gamma_i$ is singular, the standard Hasse-Weil bound does not directly apply to the singular model. Let $\tilde{\Gamma}_i$ be the normalization (smooth model) of $\Gamma_i$. There is a birational map $\nu: \tilde{\Gamma}_i \to \Gamma_i$.
The map is an isomorphism away from the singular points of $\Gamma_i$. The number of singular points is bounded by the arithmetic genus.
Crucially, the number of rational points on the singular curve is bounded by the rational points on the normalization plus the contribution from singularities (which is a small constant bound).
\end{itemize}

\Cref{thm:aubry-perret} unifies both cases using the \textit{arithmetic genus} $\pi_i$ (which equals the geometric genus for smooth curves and is strictly larger for singular curves).

For a curve of degree $d_i$ in $\PP^n$, the arithmetic genus is sharply bounded above by the \textbf{plane projection bound}: $\pi_i \le \frac{(d_i-1)(d_i-2)}{2}$.

\vspace{0.5em}
\noindent\textbf{Step 3: Aggregating the Bounds}

Summing over all components (both smooth and singular), and using the Hasse-Weil bound for singular curves as a strict upper bound for any components not defined over $\FF_Q$):
\begin{align*}
    |B_P(\FF_Q)| &\le \sum_{i=1}^m |\Gamma_i(\FF_Q)| \\
    &\le \sum_{i=1}^m \left( Q + 1 + 2\pi_i \sqrt{Q} \right) \\
    &= mQ + m + 2\sqrt{Q} \sum_{i=1}^m \pi_i
\end{align*}
The number of components $m$ is bounded by the total degree $\deg(B_P) = 9$. Since the plane projection bound is superadditive ($f(d_i+d_j) \ge f(d_i) + f(d_j)$), the sum of arithmetic genera is maximized when the degree is concentrated in a single irreducible component ($m = 1$, $d_1 = 9$):
\[ \sum_{i=1}^m \pi_i  \le \sum_{i=1}^m f(d_i) \le f\left(\sum_{i=1}^m d_i\right) = \frac{(9-1)(9-2)}{2} = 28. \]
Substituting these global bounds yields:
\[ |B_P(\FF_Q)| \le 9Q + 56\sqrt{Q} + 9 \]
This inequality relies only on the total degree and is completely uniform. It holds regardless of the choice of point $P$ and of whether the components are smooth or singular.

\vspace{0.5em}
\noindent\textbf{Step 4: Asymptotic Proof for $n$ points}

We compare the size of the combined bad locus to the total number of points on the surface.
\begin{itemize}
    \item \textbf{Surface:} By the Weil Conjectures, the number of points on a smooth cubic surface over $\mathbb{F}_Q$ is $Q^2 + Q\text{Tr}(F | H^2) + 1$, with the trace bounded below by $-2$ \cite[Table 1]{manin3}, giving a strict, universal lower bound on the total number of points: $|\tilde{V}(\mathbb{F}_Q)| \ge Q^2 - 2Q + 1$.
    \item \textbf{Combined Bad Locus:} For any set of $n$ points $\{P_1, \dots, P_n\}$, the set of points failing to be in general position with at least one $P_j$ is bounded by the union of their individual bad loci. Thus, $\big| \bigcup_{j=1}^n B_{P_j}(\mathbb{F}_Q) \big| \le n(9Q + 56\sqrt{Q} + 9)$.
\end{itemize}

We define the set of ``General Position Candidate'' for the set of $n$ points as $G = \tilde{V}(\mathbb{F}_Q) \setminus \bigcup_{j=1}^n B_{P_j}(\mathbb{F}_Q)$. Its size is bounded below by:
\[
|G| \ge (Q^2 - 2Q + 1) - n(9Q + 56\sqrt{Q} + 9)
\]
Consider the tower of quadratic extensions where $Q_j = |\tilde{k}|^{2^j}$. As $j \to \infty$, $Q_j \to \infty$. Clearly, for sufficiently large $Q_j$:
\[
Q_j^2 > (2 + 9n)Q_j + 56n\sqrt{Q_j} + 9n - 1
\]
Because $n$ and the constants are fixed, the $Q_j^2$ term strictly dominates. The threshold for $Q_j$ being "sufficiently large" depends strictly on $n$ and is absolutely uniform for any choice of $n$ points.

Thus, by passing high enough up the tower of quadratic extensions (determined strictly by $n$), we guarantee $|G| > 0$. Therefore, there exists at least one point $P' \in \tilde{V}(\tilde{K})$ such that $P'$ is in general position with all $n$ points $P_j$ simultaneously.

This completes the proof of \Cref{thm:genposall}.
\end{proof}

\begin{lemma}
\label{thm:biject}
Let $i \in \{2, 3\}$. Suppose every point on $\tilde{V}(\tilde{k})$ is $i$-class free. Then there is a bijection between the equivalence classes $V(k)/U_i$ and the equivalence classes $\tilde{V}(\tilde{k})/\tilde{U}_i$.
\end{lemma}

\begin{proof}
We will construct well-defined, mutually inverse maps between the sets of equivalence classes $V(k)/U_i$ and $\tilde{V}(\tilde{k})/\tilde{U}_i$.

\vspace{0.5em}
\noindent\textbf{Step 1: Pushing $U_i$ down to $\tilde{V}(\tilde{k})$}

Because every point $\tilde{P} \in \tilde{V}(\tilde{k})$ is $i$-class free, all lifts $P \in \pi^{-1}(\tilde{P})$ belong to the exact same $U_i$-class. We may thus define a map $f: \tilde{V}(\tilde{k}) \to V(k)/U_i$ by setting $f(\tilde{P}) = [P]_{U_i}$, where $P \in V(k)$ is any lift of $\tilde{P}$. 

We define an equivalence relation $\sim_\alpha$ on $\tilde{V}(\tilde{k})$ by:
\[ \tilde{A} \sim_\alpha \tilde{B} \iff f(\tilde{A}) = f(\tilde{B}). \]
We claim $\sim_\alpha$ is an admissible equivalence on $\tilde{V}(\tilde{k})$. We verify this by showing it satisfies the binary operation rules described in Fact \ref{item}
\begin{itemize}
    \item \textit{Secants:} Let $\tilde{P}_1 \neq \tilde{P}_2$ with no line on $\tilde{V}$ through them. They define a secant line $\tilde{L}$ intersecting $\tilde{V}$ at a third point $\tilde{P}_3$. By Lemma \ref{thm:swd-14} (Swinnerton-Dyer), we can lift the collinear triple $(\tilde{P}_1, \tilde{P}_2, \tilde{P}_3)$ on $\tilde{L} \cap \tilde{V}$ to a collinear triple $(P_1, P_2, P_3)$ on a lifted line $L$. Since $U_i$ is admissible on $V(k)$, $[P_3]_{U_i} = [P_1]_{U_i} \circ [P_2]_{U_i}$. Thus, the $\sim_\alpha$-class of $\tilde{P}_3$ is uniquely determined by the $\sim_\alpha$-classes of $\tilde{P}_1$ and $\tilde{P}_2$.
    \item \textit{Lines on $\tilde{V}$:} If $\tilde{P}_1, \tilde{P}_2$ lie on a line $\tilde{L}$ belonging to $\tilde{V}$, \Cref{thm:liftlines} states that $\tilde{L}$ lifts to a $k$-rational line $L$ belonging to $V$. The admissibility of $U_i$ forces $L \cap V(k)$ to collapse into a single $U_i$-class. Consequently, all points on $\tilde{L}$ map under $f$ to this single class, effectively collapsing $\tilde{L}$ into a single $\sim_\alpha$-class.
    \item \textit{Tangent Sections:} If $\tilde{P}_1 = \tilde{P}_2$, consider the tangent section $\tilde{C} = \tilde{V} \cap T_{\tilde{P}_1}\tilde{V}$. This configuration lifts to a tangent section $C = V \cap T_{P_1}V$ at a lifted point $P_1$. By the admissibility of $U_i$, all points on $C \setminus \{P_1\}$ share the same $U_i$-class. Their reductions onto $\tilde{C} \setminus \{\tilde{P}_1\}$ therefore map to the same class under $f$, and are thus $\sim_\alpha$-equivalent.
\end{itemize}

Furthermore, $\sim_\alpha$ explicitly inherits the $i$-th algebraic property from $U_i$ via the map $f$:
\begin{itemize}
    \item \textbf{If $i = 3$:} For any class $[\tilde{A}]_{\sim_\alpha}$, the operation $[\tilde{A}]_{\sim_\alpha} \circ [\tilde{A}]_{\sim_\alpha}$ corresponds to $f(\tilde{A}) \circ f(\tilde{A}) = [A]_{U_3} \circ [A]_{U_3} = [A]_{U_3} = f(\tilde{A})$. Thus, $[\tilde{A}]_{\sim_\alpha} \circ [\tilde{A}]_{\sim_\alpha} = [\tilde{A}]_{\sim_\alpha}$.
    \item \textbf{If $i = 2$:} For any class $[\tilde{A}]_{\sim_\alpha}$, the operation $[\tilde{A}]_{\sim_\alpha} \circ [\tilde{A}]_{\sim_\alpha}$ corresponds to $f(\tilde{A}) \circ f(\tilde{A}) = [A]_{U_2} \circ [A]_{U_2} = X_0$. Thus, there is a fixed class $\tilde{X}_0 \in \tilde{V}(\tilde{k})/\sim_\alpha$ such that $[\tilde{A}]_{\sim_\alpha} \circ [\tilde{A}]_{\sim_\alpha} = \tilde{X}_0$.
\end{itemize}
Since $\sim_\alpha$ is an admissible equivalence satisfying Property $i$, and $\tilde{U}_i$ is defined as the \emph{finest} admissible equivalence satisfying Property $i$ on $\tilde{V}(\tilde{k})$, $\tilde{U}_i$ must refine $\sim_\alpha$:
\[ \tilde{A} \sim_{\tilde{U}_i} \tilde{B} \implies \tilde{A} \sim_\alpha \tilde{B} \implies f(\tilde{A}) = f(\tilde{B}) \implies [A]_{U_i} = [B]_{U_i}. \]
This yields a well-defined mapping $\Phi: \tilde{V}(\tilde{k})/\tilde{U}_i \to V(k)/U_i$ given by $\Phi([\tilde{P}]_{\tilde{U}_i}) = [P]_{U_i}$.

\vspace{0.5em}
\noindent\textbf{Step 2: Pulling $\tilde{U}_i$ up to $V(k)$}

Conversely, we define an equivalence relation $\sim_\beta$ on $V(k)$ via the reduction map:
\[ A \sim_\beta B \iff \tilde{A} \sim_{\tilde{U}_i} \tilde{B}. \]
We verify that $\sim_\beta$ is an admissible equivalence on $V(k)$. Let $(A_1, A_2, A_3)$ and $(B_1, B_2, B_3)$ be collinear triples in $V(k)$ such that $A_1 \sim_\beta B_1$ and $A_2 \sim_\beta B_2$. Because the reduction of a line in $\mathbb{P}^3_k$ is a line in $\mathbb{P}^3_{\tilde{k}}$, and intersection multiplicities are preserved due to the smoothness of $\tilde{V}$, the reductions $(\tilde{A}_1, \tilde{A}_2, \tilde{A}_3)$ and $(\tilde{B}_1, \tilde{B}_2, \tilde{B}_3)$ are collinear triples in $\tilde{V}(\tilde{k})$. 
Since $\tilde{U}_i$ is an admissible equivalence, it preserves collinearity, which implies $\tilde{A}_3 \sim_{\tilde{U}_i} \tilde{B}_3$. 
This immediately gives $A_3 \sim_\beta B_3$. Thus, $\sim_\beta$ is an admissible equivalence on $V(k)$.

Moreover, $\sim_\beta$ inherits the respective $i$-th algebraic property from $\tilde{U}_i$:
\begin{itemize}
    \item \textbf{If $i = 3$:} The operation $A \circ A$ modulo $\sim_\beta$ traces to $\tilde{A} \circ \tilde{A}$ under $\tilde{U}_3$. By definition, $[\tilde{A} \circ \tilde{A}]_{\tilde{U}_3} = [\tilde{A}]_{\tilde{U}_3}$, so $A \circ A \sim_\beta A$.
    \item \textbf{If $i = 2$:} The relation dictates $[\tilde{A} \circ \tilde{A}]_{\tilde{U}_2} = \tilde{X}_0$. Pick any point in the class $\tilde{X}_0$ and let $X_0$ be a lift of that point. Then for all $A \in V(k)$, we have $A \circ A \sim_\beta X_0$.
\end{itemize}
Because $U_i$ is the \emph{finest} admissible equivalence on $V(k)$ satisfying Property $i$, $U_i$ must refine $\sim_\beta$:
\[ A \sim_{U_i} B \implies A \sim_\beta B \implies \tilde{A} \sim_{\tilde{U}_i} \tilde{B}. \]
This yields a well-defined mapping $\Psi: V(k)/U_i \to \tilde{V}(\tilde{k})/\tilde{U}_i$ given by $\Psi([P]_{U_i}) = [\tilde{P}]_{\tilde{U}_i}$.

\vspace{0.5em}
\noindent\textbf{Step 3: Establishing the Bijection}

The maps $\Phi$ and $\Psi$ are mutually inverse:
\begin{itemize}
    \item For any $[P]_{U_i} \in V(k)/U_i$, we have $\Phi(\Psi([P]_{U_i})) = \Phi([\tilde{P}]_{\tilde{U}_i}) = [P']_{U_i}$, where $P'$ is some lift of $\tilde{P}$. Because $\tilde{P}$ is $i$-class free, all of its lifts belong to the same $U_i$-class, so $[P']_{U_i} = [P]_{U_i}$.
    \item For any $[\tilde{P}]_{\tilde{U}_i} \in \tilde{V}(\tilde{k})/\tilde{U}_i$, we have $\Psi(\Phi([\tilde{P}]_{\tilde{U}_i})) = \Psi([P]_{U_i}) = [\tilde{P}]_{\tilde{U}_i}$.
\end{itemize}
Therefore, $\Phi$ and $\Psi$ form a strict bijection between $V(k)/U_i$ and $\tilde{V}(\tilde{k})/\tilde{U}_i$.
\end{proof}

Now, we prove Theorem \ref{thm:main}:

\MainResultOne*

\begin{proof}
\label{proof:thm-a}
Suppose there exists a tower of quadratic extensions, $\tilde{K}/\tilde{k}$, such that the reduced surface $\tilde{V}(\tilde{K})$ contains a rational line. Let $K/k$ be the corresponding tower of quadratic extensions.
Then by \Cref{thm:liftlines}
this line lifts, and via \Cref{thm:line-implies-3-torsion-free}, universal equivalence over the extension is 3-torsion-free trivial ($\mathcal{M}_U(K)[3] \cong 1$).

Else, suppose not. Then by \cite[\S 5]{swinnertondyer} the original reduction $\tilde{V}(\tilde{k})$ is isomorphic to one of three line-free cases, which we abbreviate as $n = 1$, $n = 3$, and $n = 9$. For the case $n = 1$, recall it is the unique case of an all-Eckardt reduction on which $\tilde{H}^*$ does not identically vanish (see \cite[\S 1]{swinnertondyer}). Hence, any finite extension $K$ is no longer an exception to \Cref{thm:swd-2} and therefore $V(K)$ is trivial for any finite $K$ over $k$. 

Next, note that the number of Eckardt points on smooth cubic surfaces is at most 45 in any characteristic (e.g., \cite[Lemma 20.2.7]{hirschfeld-lines}). By passing to the tower of quadratic extensions $K$ over $k$ given by \Cref{thm:genposall} for $n=2$, the reduced surface $\tilde{V}(\tilde{K})$ acquires non-Eckardt points while still containing no lines. Because the original reduction $\tilde{V}(\tilde{k})$ consisted entirely of Eckardt points, we may choose one such point $\tilde{E} \in \tilde{V}(\tilde{k})$. Over $\tilde{K}$, the point $\tilde{E}$ is still Eckardt and by \Cref{thm:genposall} is in general position with some other point. Therefore, by \Cref{thm:class-free}, the universal equivalence classes of the lifts of $\tilde{E}$ form a CML of period 2. Consequently, the 3-torsion among its lifts is trivial, meaning $\tilde{E}$ is 3-class-free. (Note this relies strictly on a $p$-adic volumetric argument and does not require $\tilde{E}$ to lift to a rational Eckardt point in $V(K)$).

We now propagate this property to the entire surface. First, recall the secant lifting property \cite[remark before Lemma 15]{swinnertondyer}: if a point $\tilde{A}$ is 3-class-free, and is in general position with $\tilde{B}$ (yielding a third distinct transversal intersection $\tilde{C}$), we may fix a single lift $C \in V(K)$ of $\tilde{C}$. Then for any lift $B$ of $\tilde{B}$, the line $BC$ intersects $V$ at a third point $A_B$ lifting $\tilde{A}$. Because $\tilde{A}$ is 3-class-free, $[A_B]_{U_3}$ is a constant class. Thus $[B]_{U_3} = [A_B]_{U_3} \circ [C]_{U_3}$ is also constant, making $\tilde{B}$ 3-class-free. Therefore, to prove a point is 3-class-free, it suffices to find a single 3-class-free point in general position with it.

Let $\tilde{P} \in \tilde{V}(\tilde{K})$ be an arbitrary point. Applying \Cref{thm:genposall} to the pair of points $\{\tilde{E}, \tilde{P}\}$, there exists a point $\tilde{Q} \in \tilde{V}(\tilde{K})$ that is in general position with both $\tilde{E}$ and $\tilde{P}$ simultaneously. Because $\tilde{Q}$ is in general position with $\tilde{E}$, the point $\tilde{Q}$ becomes 3-class-free. Because $\tilde{P}$ is in general position with $\tilde{Q}$, the point $\tilde{P}$ also becomes 3-class-free. Since $\tilde{P}$ was arbitrary, every single point on $\tilde{V}(\tilde{K})$ is 3-class-free. This fulfills the strict hypothesis of \Cref{thm:biject}, providing the bijection $V(K)/U_3 \cong \tilde{V}(\tilde{K})/\tilde{U}_3$. Since $\tilde{V}(\tilde{K})$ contains no lines, $\tilde{V}(\tilde{K})/\tilde{U}_3$ is trivial by \Cref{thm:swd-1}.\footnote{A brief sketch of a possible proof first appeared in the preprint: D. Kanevsky, ``Some remarks on Brauer equivalence for cubic surfaces,'' Max Planck Inst. für Math. (1984). While other results from that preprint were cited in Manin’s \textit{Cubic Forms}, this specific problem was not included among the solved cases discussed therein. We believe that the absence of formal proofs regarding the existence of general position upon extension, as well as the bijection of 3-universal equivalence classes under modulo 2 reduction, constituted significant gaps. Consequently, the problem has effectively remained open until now.} Hence, we have a tower of quadratic extensions $K$ over $k$ such that $\#\{V(K)/U_3\}$ consists of one element.

Using \Cref{thm:RU} we get
\begin{equation}
    \#\mathcal{M}_{R}(k)[3] \leq \#\mathcal{M}_{U}(K)[3]=1,
\end{equation}
i.e., the 3-torsion component of $\mathcal{M}_{R}(k)[3]$ ($R$-equivalence over the base local field), consists of one element. Then by the structure theorem for $R$-equivalence CMLs (\Cref{thm:structure}), $\mathcal{M}_{R}(k)$ is trivial or has exponent 2.
\end{proof}

\section{Proof of \Cref{single}}
\label{sec:1pt-r-equiv}

In this section we will prove Theorem \ref{single}:

\MainResultTwo*

Our proof consists of three parts. In the first part, following Section 3.4 in \cite{dimitrikanevsky},  we demonstrate the existence of an admissible equivalence on this surface consisting of exactly two classes. In the second part, we show that any point on a transformed surface possessing two irrational tangent lines in its reduction modulo 2 is class-free. This immediately implies that the remaining points on the tangent curve—defined by the intersection of the tangent plane and the transformed cubic surface modulo 2—are also class-free. Finally, we develop a novel $R_2$-equivalence method and prove that it is trivial.

\subsection{The Surface with One-Point Reduction}
\label{sec:trans}

First, one can confirm that the reduction of \Cref{eq:main} is isomorphic to the ``unique'' single-point reduced surface in \cite[5.5]{swinnertondyer}.

We begin as in \cite[\S 3.1]{dimitrikanevsky} by transforming \Cref{eq:main} from $F(T_0, T_1, T_2, T_3)$ into $F_1(T_0, T'_1, T'_2, T'_3)$ via a change of variables, followed by division by the factor $2^{3}$:
\begin{equation}
\label{psi}
\Phi_1:    T_0 \to T_0, \quad T_1 \to 2^{3}T'_1, \quad T_2 \to 2T'_2, \quad T_3 \to 2T'_3
\end{equation}
The transformed polynomial $F_1$ is obtained by dividing the sum of the expanded terms by the factor $2^3$. We group the resulting terms by the remaining power of 2.
\begin{align}
F_1(T_0, T'_1, T'_2, T'_3) &= \frac{F(T_0, 2^{3}T'_1, 2T'_2, 2T'_3)}{2^{3}} \nonumber\\
&= T_0^2T_1' + T_0(b_0T_2'^2+b_1T_2'T_3' + b_2  T_3'^2) + \left((T'_2)^3 + (T'_2)^2T'_3 + (T'_3)^3\right) \nonumber\\ 
&\quad + 2 ( \text{terms in } T_0, T'_1, T'_2, T'_3 )
\label{Fm}
\end{align}
This equation defines a projective cubic surface 
\begin{equation}
\label{V1}
V_1: F_1(T_0, T'_1, T'_2, T'_3)=0
\end{equation}
such that modulo 2 it defines the cubic surface
\begin{equation}
\label{eq:v1}
\tilde V_1:  \tilde T_0^2 \tilde T_1' + \tilde T_0(\tilde T_2'^2+\tilde T_2'\tilde T_3' +   \tilde T_3'^2) + (\tilde T'_2)^3 + (\tilde T'_2)^2\tilde T'_3 + (\tilde T'_3)^3 = 0.
\end{equation}

\label{adm}
Next, we define a partition of the set $\tilde{S} \subset \tilde{V}_1(\mathbb{F}_2)$ based on the value of the quadratic form $Q(y,z) = y^2+yz+z^2$.
\begin{itemize}
    \item \textbf{Class $\tilde X_0$}: Points where the quadratic form vanishes ($Q=0$):
    \begin{equation}
        \tilde X_0 = \{ \tilde{P} \in \tilde{S} \mid Q(\tilde T'_2, \tilde T'_3) = 0 \} = \{ (1,0,0,0) \},
    \end{equation}
    \item \textbf{Class $\tilde X_1$}: Points where the quadratic form is non-zero ($Q=1$):
    \begin{equation}
        \tilde X_1 = \{ \tilde{P} \in \tilde{S} \mid Q(\tilde T'_2, \tilde T'_3) = 1 \} = \{ (1,0,1,0), (1,0,1,1), (1,0,0,1) \}.
    \end{equation}
\end{itemize}
We define an equivalence relation $\mathcal{A}$ on the original $V(k)$ by lifting these classes.
Two points $P, Q \in V(k)$ are $\mathcal{A}$-equivalent if their reductions (after the transformation $\Phi_1$) fall into the same set $X_i$.
We denote as class $X_i \subset V_1(k)$ the a set of points on $V_1(k)$ lying above $\tilde X_i$.
One verifies that $\tX_0, \tX_1$ are the universal equivalence classes of the reduced surface such that $\tX_0 \circ \tX_0 = \tX_1$ and $\tX_1 \circ \tX_1 = \tX_1$.
Pulling back to the surface $V$ over 2-adic $k$, \cite{dimitrikanevsky} showed that $X_0, X_1$ define a period-2 component in the universal equivalence of $V(k)$.

Next, we  observe that per \Cref{thm:main}, we know $R$-equivalence on $V(k)$ is 3-torsion-free.
Now, we will prove that $R$-equivalence on $V(k)$ is 2-torsion-free in the next subsection.

\subsection{Finishing Proof of Theorem \ref{single}}
\begin{lemma}
\label{thm:locus}
    Let $\tilde W$ be the projective cubic surface over $\mathbb{F}_2$ defined by the homogeneous equation:
    \begin{equation}
    \label{tildeG}
        \tilde G(X, Y, Z, T) = X^2T + X(Y^2 + YZ + Z^2) + Y^3 + Y^2Z + Z^3 = 0.
    \end{equation}
    Let $\tilde{P}_0 = (1:0:1:0)$ and let $\tilde{S}$ be the locus of points $Q \in \tilde W(\mathbb{F}_2)$ such that $\tilde{P}_0$ lies on the tangent plane $\Pi(Q)$. Then the point $\tilde{P} = (1:0:0:0)$ is a non-singular point of $\tilde{S}$.
\end{lemma}

\begin{proof}
We follow the method in Lemma 16 of \cite{swinnertondyer}.
Take coordinates such that $\tilde P = (1,0,0,0)$, $\tilde P_0 = (0,1,0,0)$ and $\Pi(\tilde P)$ is $T = 0$. Then the equation of $\tilde W$ is
\begin{equation}
\label{f}
    f = X^2T + \text{terms at most linear in X} = 0 
\end{equation}
and $f$ has no terms in $Y^3$. The curve $\tilde S$ has equation $F_Y = 0$ where the subscript denotes differentiation.  $\tilde S$ has singular point at $\tilde P$ if and only if
\begin{equation}
    \label{fxy}
    f_{XY} = f_{YY} = f_{YZ} = 0 \text{ at } \tilde P
\end{equation}
Since char $\mathbb{F}_2$ = 2 this means that $\tilde P$ is singular on $\tilde S$ if $f$ has no 
term in $XYZ$, i.e. $\tilde P$ is a cusp on $\Gamma(\tilde P)$. But one can immediately check that
$\tilde P = (1, 0, 0, 0)$ is not a cusp.
\end{proof}
\begin{lemma}
\label{thm:one-pt-class-free}
Let $W$ be the cubic surface over $\mathbb{Q}_2$ defined by
\begin{equation}
    \label{G}
        \tilde G(X, Y, Z, T) = X^2T + X(Y^2 + YZ + Z^2) + Y^3 + Y^2Z + Z^3 + 2(...)= 0.
\end{equation}
Then $\tilde{P} = (1, 0, 0, 0) \in \tilde W(\mathbb{F}_2)$ is class free.
\end{lemma}
\begin{proof}
We follow the method in Lemma 15 and Lemma 17 of \cite{swinnertondyer}.

\vspace{0.5em}
\noindent\textbf{Step 1:}

Let $P = (1, 0, 0, 0)$ and $P'= (1, y, z, t)$ be distinct points of $W$  above $\tilde{P}$.
The intersection of the two tangent planes $\Pi(P) \cap \Pi(P')$ within $T=0$ yields a line $L$ that modulo 2 belongs to the tangent plane at  $\tilde{P}$. It can be shown as in the following. Because $P'$ is close to $P$ in the $2$-adic topology, substituting the partial derivatives of $G$ to the first order in $\mathfrak{p}$-adic valuations shows this line is governed by:
\begin{equation}
Y G_Y(P') + Z G_Z(P') = 0
\end{equation}
Evaluating the formal partial derivatives, we find $G_Y(P') \approx z$ and ${F_1}_Z(P') \approx y$ modulo higher-order terms. By symmetry, we may assume $v(y) \ge v(z)$. We have $v(t) > v(z)$. The intersection line $L$ thus takes the form $Y + (y/z)Z + \dots = 0$, i.e.  the line $L$ that modulo 2 belongs to the tangent plane at  $\tilde{P}$. Hence $\tilde L \mod 2$ meets $\tilde W$ again in a rational point $\tilde P_1 \neq \tilde P_2$, and so $L$ meets $W$ in a rational point $P_1$ above $\tilde P_1$.

\vspace{0.5em}
\noindent\textbf{Step 2:}

Let $C$ be the class of $P_1$. Since all rational points of $\Gamma(P)$ except perhaps $P$ belong to the same class, all rational points of $\Gamma(P)$ and $\Gamma(P')$ except perhaps $P$ and $P'$ belong to $C$. Let $P_2 $ be any rational point on $W$ such that $\tilde P_2 = (1,0,1,0) \in \Gamma (\tilde P)$. By \Cref{thm:locus} there exists a rational $P'$  on $W$ above $\tilde P$ such that $P_2$ lies on $\Gamma(P')$; hence $P_2 \in C$.\\
Now let $L_1$ denote the line $PP'$. A calculation like in Step 1 shows that $\tilde L_1$ is a rational line through $\tilde P$ in $\Pi(\tilde P)$. Therefore, as for $L$ above, $L_1$ meets $W$ again in a rational point $P_2$ of the kind described above. Now $P, P, P_1$ and $P, P', P_2$ are collinear. But $P_1 \sim P_2$. Therefore $P \sim P'$.
\end{proof}

\subsubsection{$R$-equivalence for Cubics Reducing to One Point}
To prove the $R$-equivalence part of \Cref{single} we need the following statement:
\begin{lemma}
\label{thm:tanglimit}
Let $V$ be a cubic surface over a local $p$-adic field $k$. Let $P \in V(k)$ be a point such that its reduction $\tilde{P} = P \pmod{\mathfrak{p}}$ is a smooth point on $\tilde{V} = V \pmod p$, and $\tilde{V}$ contains no lines passing through $\tilde{P}$.

Let $K$ be a separable quadratic extension of $k$. Let $\{Q_i\}_{i \ge 1}$ be a sequence of points in $V(K) \setminus V(k)$ converging to $P$ in the $p$-adic topology. Let $Q'_i$ be the Galois conjugate of $Q_i$ over $k$. Let $r_i$ be the third point of intersection of the line $L_i$ passing through $Q_i$ and $Q'_i$ with $V$.

Then the sequence of lines $L_i$ converges to a tangent line to $V$ at $P$ defined over $k$, and the sequence of points $r_i$ converges to the intersection of this tangent line with $V$ (specifically to the point $R$ such that the intersection cycle is $2P+R$).
\end{lemma}

\begin{proof}
We proceed in steps:

\vspace{0.5em}
\noindent\textbf{Step 1: Local Coordinates and Smoothness}

    Since $\tilde{P}$ is a smooth point of $\tilde{V}$, $P$ is a smooth point of $V$ (by Hensel's Lemma/lifting of the non-vanishing gradient). We can choose affine coordinates $(x,y,z)$ defined over the ring of integers of $k$ such that $P$ is the origin $(0,0,0)$ and the tangent plane $T_P V$ is given by $z=0$.
    The equation of the cubic surface $V$ can be written as:
    \[ F(x,y,z) = F_1(x,y,z) + F_2(x,y,z) + F_3(x,y,z) = 0 \]
    where $F_1, F_2, F_3$ are homogeneous polynomials of degrees 1, 2, and 3 respectively. Since $z=0$ is the tangent plane, $F_1(x,y,z) = z$ (up to a unit scaling factor). Thus:
    \[ F(x,y,z) = z + F_2(x,y,z) + F_3(x,y,z) = 0 \]

\vspace{0.5em}
\noindent\textbf{Step 2: Parametrization of Points}

    Let $K = k(\sqrt{D})$ for some $D \in k \setminus k^2$. Since $Q_i \in V(K) \setminus V(k)$ and $Q_i \to P$, we can write $Q_i$ in terms of its coordinates in the basis $\{1, \sqrt{D}\}$:
    \[ Q_i = A_i + \sqrt{D} B_i \]
    where $A_i, B_i \in k \times k \times k$. The Galois conjugate is $Q'_i = A_i - \sqrt{D} B_i$.
    Since $Q_i \to P=(0,0,0)$, we have $A_i \to 0$ and $B_i \to 0$ in the $p$-adic topology.
    Since $Q_i \notin V(k)$, $Q_i \neq Q'_i$, which implies $B_i \neq 0$.

\vspace{0.5em}
\noindent\textbf{Step 3: Convergence to a Tangent Line}

    The line $L_i$ passing through $Q_i$ and $Q'_i$ is parameterized by $R(t) = A_i + t B_i$. The direction of this line is given by the vector $B_i$.
    Since $Q_i$ lies on $V$, $F(Q_i) = 0$. Using the Taylor expansion of $F$ around $A_i$:
    \[ F(A_i + \sqrt{D} B_i) = F(A_i) + \nabla F(A_i) \cdot (\sqrt{D} B_i) + \mathcal{O}(|B_i|^2) = 0 \]
    Considering the terms associated with $\sqrt{D}$ (the irrational part), and dividing by $\sqrt{D}$ (since $B_i \neq 0$):
    \[ \nabla F(A_i) \cdot B_i + \text{Higher Order Terms} = 0 \]
    As $i \to \infty$, $A_i \to P$. By the continuity of the gradient, $\nabla F(A_i) \to \nabla F(P)$. Since $B_i \to 0$, the higher order terms vanish faster than linear terms. Thus, the direction vectors $v_i = B_i / \|B_i\|$ satisfy:
    \[ \lim_{i \to \infty} \nabla F(P) \cdot v_i = 0 \]
    This implies that any limit direction of the secant lines lies in the kernel of the gradient form at $P$, which is exactly the tangent plane $T_P V$.
    Thus, the line $L_i$ converges to a line $L_\infty$ passing through $P$ and contained in the tangent plane $T_P V$. Since $L_i$ passes through $k$-rational points $A_i$ with $k$-rational direction $B_i$, the limit line is defined over $k$.

\vspace{0.5em}
\noindent\textbf{Step 4: The Third Intersection Point}

    The intersection of the line $L_i$ with $V$ is determined by the roots of the cubic polynomial $g_i(t) = F(A_i + t B_i)$. The roots are $t_1 = \sqrt{D}$ (corresponding to $Q_i$) and $t_2 = -\sqrt{D}$ (corresponding to $Q'_i$). Let $t_3$ be the parameter for the third point $r_i$.
    By Vieta's formulas, the sum of the roots satisfies $t_1 + t_2 + t_3 = - \frac{\text{coeff of } t^2}{\text{coeff of } t^3}$.
    \[ \sqrt{D} - \sqrt{D} + t_3 = t_3 = - \frac{F_2(i)}{F_3(i)} \]
    where $C_3(i)$ is the coefficient of the cubic term of $F$ restricted to $L_i$. In the limit, the line becomes a tangent line $L_\infty$.
    
    The condition that ``there is no line on $\tilde{V}$ through $\tilde{P}$'' guarantees that the restriction of the cubic form to the tangent plane is not identically zero, and specifically, the reduced tangent line is not contained in $\tilde{V}$. By Hensel's Lemma structures, this implies $L_\infty$ is not contained in $V$. Thus, the intersection cycle $L_\infty \cdot V$ is a finite 0-cycle of degree 3.
    Since $L_\infty$ is tangent at $P$, the intersection cycle is of the form $2P + R$, where $R \in V(k)$.
    By the continuity of the roots of polynomials with respect to their coefficients, the point $r_i$ converges to $R$.
    
    Since $L_\infty \subset T_P V$, the point $R$ lies on the intersection curve $C_P = V \cap T_P V$.
    Therefore, $r_i$ converges to a point on the tangent plane defined over $k$.
\end{proof}

With the lemmas established, we proceed to prove the $R$-equivalence part of \ref{single}.
Let $\mathcal{U}$ denote the universal equivalence relation on $V(k)$. The set of equivalence classes $V(k)/\mathcal{U}$ consists of two classes, denoted $X_0, X_1 \subset V(\mathbb{Q}_2)$, which satisfy the composition laws $X_0 \circ X_0 = X_1$ and $X_1 \circ X_1 = X_1$ under the collinear binary operation on the cubic surface.

\begin{proof}[Proof of $R$-equivalence in the one-point case]
Consider the unramified quadratic extension $K = \mathbb{Q}_2(\theta)$, where $\theta^2 + \theta + 1 = 0$. The points $Z_1 = (1, \theta, 0, 0)$ and $Z_2 = (1, \theta^2, 0, 0)$ in $V(K)$ are Galois conjugate, are in general position, and are not Eckardt points. Consequently, by Theorem \ref{thm:swd-2}, all points in $V(K)$ are universally equivalent.

Let $P_1$ be a point in $X_1$ and let $P_0$ be any point in $X_0$. We will prove that $P_0 \sim P_1 \pmod R$.

Consider the reduction modulo $2$. The set $X_0 \pmod 2$ consists of a single point in $\tilde{V}(\mathbb{F}_2)$, which we denote by $\tilde{P}_0$. Let $\tilde{Z}_1 = Z_1 \pmod 2$ and $\tilde{Z}_2 = Z_2 \pmod 2$. These are conjugate points in general position on $\tilde{V}(\tilde{K})$, and the triple $\tilde{P}_0, \tilde{Z}_1, \tilde{Z}_2$ is collinear on $\tilde{V}$.

For any point $P_0 \in X_0 \subset V(\mathbb{Q}_2)$ lying above $\tilde{P}_0$, we can choose a line $L$ defined over $\mathbb{Q}_2$ passing through $P_0$ such that its reduction $\tilde{L} = L \pmod 2$ passes through $\tilde{Z}_1$ and $\tilde{Z}_2$. Let the intersection of $L$ with $V(K)$ be the cycle $P_0 + H_1 + H_2$, where $H_1, H_2 \in V(K)$. The points $H_1$ and $H_2$ are conjugate over $\mathbb{Q}_2$ and are in general position.

Since $V(K)$ consists of a single universal equivalence class, there exists a chain of rational curves connecting $H_1$ to $P_1$ (viewed as a point in $V(K)$). We follow the method of Manin described in \cite[Section 15.1.3]{manin3}. There exists a sequence of points $x_0, x_1, \dots, x_{r+1}$ in $V(K)$ with $x_0 = H_1$ and $x_{r+1} = P_1$, such that these points are connected by $K$-morphisms
\[
f_i: \mathbb{P}_K^1 \to V \otimes K, \quad \text{with } f_i(0) = x_i \text{ and } f_i(\infty) = x_{i+1}.
\]
According to \cite[\S 15.1.3.iii]{manin3}, the intermediate points $x_1, \dots, x_{r-1}$ can be chosen such that the conjugate pairs $x_i, \bar{x}_i$ in $V(K)$ are in general position for all $i = 1, \dots, r-1$.

We let the non-trivial automorphism of the extension $K/\mathbb{Q}_2$ act on $\mathbb{P}_K^1$ and $V \otimes K$ via the second factor. This action defines the conjugate morphisms $\bar{f}_i$. We define the trace morphisms
\[
g_i = f_i \circ \bar{f}_i : \mathbb{P}_{\mathbb{Q}_2}^1 \to V, \quad \text{for } i = 0, \dots, r.
\]
By definition, for all $t \in \mathbb{P}^1(\mathbb{Q}_2)$, the point $g_i(t)$ is the third intersection point of the line connecting $f_i(t)$ and $\bar{f}_i(t)$ with the surface $V$. In particular,
\[
g_i(0) = x_i \circ \bar{x}_i, \quad g_i(\infty) = x_{i+1} \circ \bar{x}_{i+1}.
\]
Since $x_{r+1} = P_1 \in V(\mathbb{Q}_2)$, we analyze the final morphism $g_r$. It is a rational map from $\mathbb{P}^1(\mathbb{Q}_2)$ to $V(\mathbb{Q}_2)$, defined as a morphism on a Zariski open set $U \subset \mathbb{P}^1(\mathbb{Q}_2)$.

 In the 2-adic topology, we can find points in $U$ sufficiently close to $\infty \in \mathbb{P}_{\mathbb{Q}_2}^1$ such that their images $P'$ under $g_r$ are close to $Q' \in X_1$. Consequently (from \Cref{thm:tanglimit}), for some neighborhood $W'$ of $\infty$ in $\PP_{\Q_2}^1$ under the 2-adic topology, the images of all points $P'$ in $W$ under $g_r$ map into the class $X_1$. This implies that the rational curve defined by $g_r$ connects the component containing $x_r \circ \bar{x}_r$ to a point in $X_1$.

Tracing back through the chain (starting from $g_0(0) = H_1 \circ \bar{H}_1 = P_0$), we conclude that $P_0$ is $R$-equivalent to a point in $X_1$. Since $P_0$ was arbitrary, this proves the theorem.
\end{proof}

\begin{ai}
\section{Disclosure of Artificial Intelligence (AI) Use}
\label{sec:ai-assist}

Our explorations began in March 2025, and so many of our interactions predate suggested norms on cataloging and presenting AI use (e.g., \cite{schmitt2025extremal, feng2026aletheia}). A further complication that our is \textit{process} is many-to-many with our \textit{presentation}; that is, our publications derive from many overlapping interactions, with release order modulated primarily by human pedagogy, taste, and confidence in results. Hence, we defer interaction specifics (e.g., human-AI interaction cards \cite{feng2026aletheia}) and overall lessons learned to a \textit{unified} companion report \cite{aireport}, where we frame this paper as one of many within our AI-assisted effort to fill out the theory of admissible equivalences on cubic surfaces.

Exposition and definitions are mostly human-written. Proofs of new lemmas are largely written by large language models (LLMs), Gemini 3 Pro and Deep Think, but under human direction and editing.
Specifically:
\begin{itemize}
\item In understanding why \cite{manin3} did not mention the claimed proof of the unpublished preprint \cite{dkbrauer} despite referencing its other results, we identified the assumptions of sufficient points in general position and the bijection of equivalence classes upon lifting as critical gaps. We conceived that both could be closed with standard methods, but the exact methods and their rigorous use (\Cref{thm:genposall,thm:biject}) were done by Gemini 3 Deep Think.
\item \Cref{thm:locus} and \Cref{thm:one-pt-class-free} are based on \cite[Lemma 15-17]{swinnertondyer}. They were respectively verified by Gemini 3 Pro and drafted by Gemini 3 Deep Think.
\item \Cref{thm:tanglimit} was written by Gemini 3 Deep Think with manual guidance. The authors would/could not have made this proof rigorous to the level attained by AI.
\end{itemize}

Below is the timeline of interactions resulting in this work. Note that the extended period does not necessarily reflect the intrinsic effort required or capability of named models, as (i) this work is a secondary project by AI researchers at Google DeepMind who happened to have mathematical training a decade ago, and (ii) this is only the first paper derived from the time period:

\begin{itemize}
  \item \textbf{March 2025:} Our investigation began with our interest in finding the $R$-equivalence of the \textit{3-adic} version of Manin's surface, i.e., over $\mathbb{Q}_3(\theta)$ (bad reduction). The non-associative universal equivalence on this surface was only recently established by \cite{dknonassoc, dknonassoc2}. The authors felt that AI tools had matured sufficiently to be useful for this problem. (In a follow-up work we will show that $R$-equivalence is also equal to Brauer equivalence in this case.)
  \item \textbf{April 2025:} For LLM use, we acquired and digitized copies of works such as \cite{dimitrikanevsky} and \cite[\S 8]{swinnertondyer}, who gave a construction (a curve given by intersecting a cubic and quadric) demonstrating trivial $R$-equivalence over $\mathbb{F}_2(\theta)$. Under incorrect intuition that this construction lifted to $\mathbb{Q}_2(\theta)$, we wanted to find a similar construction over $\mathbb{F}_3(\theta)$ that lifted to $\mathbb{Q}_3(\theta)$. We used AlphaEvolve \cite{alphaevolve} to automate the search for valid intersections with the right genus, which worked but quickly plateaued, leading us to reconsider our assumption. Gemini 2.5 Pro \cite{gemini} devised an argument for why the characteristic 2 construction could not lift as-is. (In Feb.\ 2026 we would finally get a copy of \cite{weak-approx}, who admitted the same!)
  \item \textbf{May 2025:} This failure is consistent with the statement of Manin's, that we find at this point, in the final edition of \textit{Cubic Forms} \cite{manin3}---suggesting the 2-adic problem was still open. Rapid ideation and discussions with Gemini 2.5 Pro commenced, culminating in a natural language proof written by Gemini.  We were excited to share what we believed to be the first research mathematics result co-written by generative AI.
  \item \textbf{August 2025:} Gemini 2.5 Pro with web search rediscovered the first author's preprint \cite{dkbrauer} on a Max Planck Institute web server, which gave a similar approach (in intuitive detail) to Manin's question. Due to its age (40+ years ago), even the first author no longer had a copy of their work and forgot their own ``proof!'' We paused our announcement.
  \item \textbf{Sept.\ to Nov.\ 2025:} We grew to understand why Manin (despite citing \cite{dkbrauer} in \cite{manin3} for other results) and later \cite{weak-approx} might not have considered the problem solved; see footnote to the proof of \Cref{thm:main} in \Cref{proof:thm-a}. Gemini Deep Think and Gemini 3 Pro (preview) are released; due to their improved capabilities, we performed deeper scrutiny and requested levels of detail unusual for publications in algebraic geometry. Critiquing these revived our mathematical training, improved our knowledge, and gave new prompting and proof strategies.  Reports of LLM progress on research math proliferate; see \cite[\S 7]{feng2026aletheia} for a survey.
  \item \textbf{Dec.\ 2025 to Feb.\ 2026:} Using Gemini 3 Pro and Deep Think, we built atop our initial insights from Gemini 2.5 Pro and used them to write highly detailed proofs closing the gaps in \cite{dkbrauer}'s argument and thus resolving Manin's 1972 question. These LLMs also realized a program to extend our work to the non-trivial universal equivalence case in \cite[\S 3.4]{dimitrikanevsky}; their large casework and mixed success led the authors to simpler geometric reasoning whose lemmas AI  made rigorous. Together, we viewed these as a coherent set of results for release.
  \item \textbf{Mar.\ 2026:} Final writing refinements aided by Gemini 3.1 Pro and Deep Think, improving presentation correctness (e.g., for \Cref{thm:aubry-perret}, citing \cite{aubry1996weil} \cite{leep1994number} instead of a later Aubry-Perret work that restates/refines the bounds) and streamlining exposition (e.g., using \cite{kollar2003rational} to avoid ``general type'' restrictions, or skipping quasigroups along the way to defining CMLs).
\end{itemize}

We share this timeline to give a human and historical angle to our work. Due to the ambiguity about how ``essential'' AI was in our situation, and as the final proof approaches were human-driven (despite some initial versions by AI), we follow the guidance of \cite[5.1.1]{feng2026aletheia} and class this publication-level human-AI work as ``primarily human'' (H2) in their taxonomy. That said, we believe our \textit{process} demonstrates a sustained non-trivial human-AI collaboration where AI's role continues to increase, as we find in forthcoming C2-level works in this series.
\end{ai}

\bibliographystyle{amsalpha}
\bibliography{references}

\end{document}